\documentclass[12pt]{amsart}
\usepackage[margin=1in]{geometry}
\usepackage{mathrsfs}
\usepackage{latexsym}
\usepackage{float}
\usepackage{amssymb}
\usepackage[all]{xy}
\usepackage{epsfig}
\usepackage{color}
\usepackage{graphics}
\usepackage{hyperref}

\newtheorem{Thm}[equation]{Theorem}
\newtheorem{Lem}[equation]{Lemma}

\begin{document}
\baselineskip=12pt

\title{On detecting the trivial rational $3$-tangle}
\author{Bo-hyun Kwon}
\date{Monday, March 13,  2023}
\maketitle
\begin{abstract} 
An important issue in classifying the rational $3$-tangle is how to know whether or not the given tangle is the trivial rational 3-tangle called $\infty$-tangle. The author\cite{1} provided a certain algorithm to detect the $\infty$-tangle. In this paper, we give a much simpler method to detect the $\infty$-tangle by using the $\textit{bridge arc replacement}$. We hope that this method can help  prove many application problems such as a classification of $3$-bridge knots.
\end{abstract}

\section{Introduction}

A classification of rational $2$-tangle is provided by  Conway~\cite{0} in 1970. It is a perfect classification of them since he had shown that each rational $2$-tangle is presented by the corresponding rational number. The author~\cite{1} gave an algorithm to classify the rational $3$-tangles. Unfortunately, the algorithm is not enough to classify the rational $3$-tangle completely since it is only used to distinguish  given two rational $3$-tangles not assigning them to a certain natural group. The main part of the algorithm is to show whether or not the modified rational $3$-tangle from the given two rational $3$-tangles is the trivial rational $3$-tangle called  the $\infty$-tangle. There is a basic way to detect $\infty$-tangle by the fundamental group argument, but the algorithm is much better than the traditional way. In this paper, we provide a new method to detect $\infty$-tangle.

\section{The trivial rational $3$-tangle, $\epsilon$}

 \begin{figure}[htb]
\includegraphics[scale=.35]{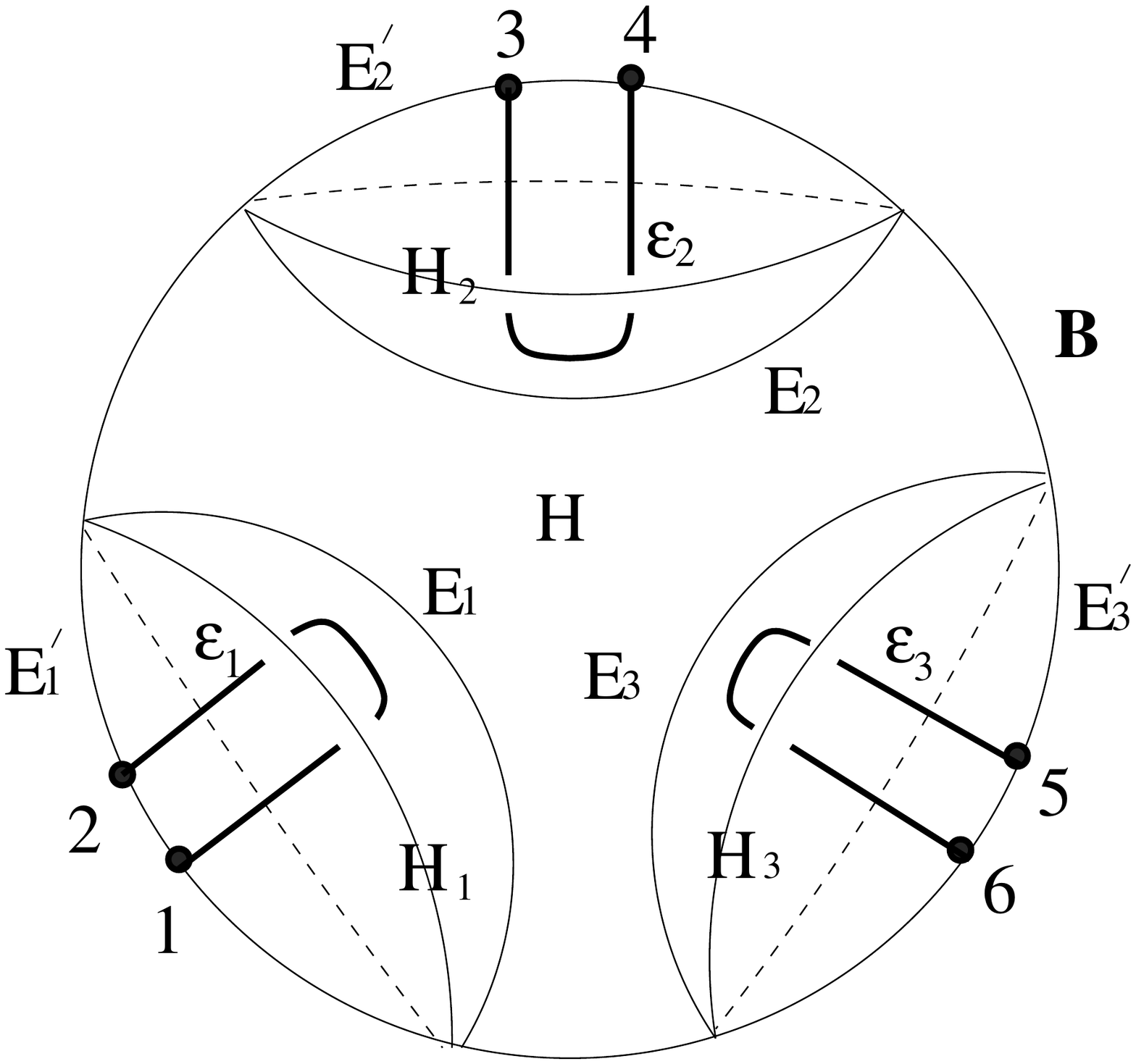}
\vskip -100pt
\caption{ Three compressing disks in $\infty$-tangle, $(B,\epsilon)$}
\label{F1}
\end{figure}
Consider an ordered pair $T=(B,\epsilon=\epsilon_1\cup\epsilon_2\cup\epsilon_3)$ as in Figure~\ref{F1}, where $B$ is the unit $3$-ball.
This trivial rational $3$-tangle  in  $B$ is called an $\infty$ tangle  or $\epsilon$.
 Let $E_1$,$E_2$ and $E_3$ be the three  pairwise disjoint, non-parallel compressing disks in $B-\epsilon$ as in Figure~\ref{F1}.
Then $E_1$,$E_2$ and $E_3$ separate $B$ into four components. Let $H_i$  be the component which contains $\epsilon_i$ and $H=cl(B-(H_1\cup H_2\cup H_3))$, where $cl(*)$ means that the closure of $*$.
Let $E_i'$ be the 2-punctured disk in $\Sigma_{0,6}$ so that $\partial E_i'=\partial E_i$ and $E_i\cup E_i'$ bounds the ball $H_i$ in $B$.
Let  $E=E_1\cup E_2\cup E_3$,  $E'=E_1'\cup E_2'\cup E_3'$ and $\partial E=\partial E'=\partial E_1\cup \partial E_2\cup \partial E_3$. Let $\Sigma_{0,6}=\partial B-\epsilon$ and $P=cl(\Sigma_{0,6}-(E'))$ which is a pair of pants.

\section{Bridge arc replacement}

 Let $\tau=\tau_1\cup\tau_2\cup\tau_3$ be a rational $3$-tangle in $B$, where $\tau_i$ be a properly embedded arc in $B$.  A  disk $D$ in $B$ is called a \textit{bridge disk} if $D^\circ\cap \tau=\emptyset$ and $\partial D=\tau_i\cup\beta$ and $\tau_i\cap \beta=\partial \tau_i$ for some $i\in\{1,2,3\}$, where $\beta$ is a simple arc in $S^2$ so that $\beta^\circ\cap \tau=\emptyset$. The simple arc $\beta$ in $\Sigma_{0,6}$ is called a \textit{bridge arc} if  it cobounds a bridge disk with a component of $\tau$. We note that there is a collection of  disjoint three bridge arcs for a rational $3$-tangle. However, we also note that the collection is not unique (up to isotopy)  for the same rational $3$-tangle. Actually, there are infinitely many different bridge arcs for the same string of the rational $3$-tangle. We say that the simple closed curve $\gamma$ is \textit{obtained from} the bridge arc $\beta$ if the boundary of a small regular neighborhood of $\beta$ in $\Sigma_{0,6}$ is isotopic to $\gamma$.
 
\begin{Lem}\label{L1}
Let $\beta_1$ and $\beta_2$ be two disjoint  bridge arcs on $\Sigma_{0,6}$ of $(B,\tau)$, where $\tau$ is a rational $3$-tangle in $B$. Let $\beta_3$ be an arc connecting the two remaining punctures on $\Sigma_{0,6}$ which is disjoint with $\beta_1\cup\beta_2$. Then $\gamma_3$ is also a bridge arc on  $\Sigma_{0,6}$ of $(B,\tau)$.
\end{Lem}
\begin{proof}
We first claim that there are two disjoint bridge disks so that each of them contains the bridge arcs $\beta_1$ and $\beta_2$ respectively. We take a bridge disk  bounded by the union of $\beta_1$ and a string of $\tau$. Then cut $B$ along the bridge disk and open it along the two cut disks to have a rational $2$-tangle. Since $\beta_2$ is disjoint with $\beta_1$, we can take a bridge disk bounded by $\beta_2$ and a string of the obtained rational $2$-tangle. We glue the two cut disks back to have the original rational $3$-tangle. This implies the existence of two disjoint bridge disks satisfying the condition. Moreover, from the rational $2$-tangle, if we cut the tangle along the bridge disk bounded by $\beta_2$ and a string then we have a rational $1$-tangle. It is clear that any arc connecting the two endpoint of the $1$-tangle on $\Sigma_{0,2}$ without intersecting the opened two disks along the bridge disks cobounds a bridge disk in the $3$-ball containing the rational $1$-tangle. By gluing the two pairs of bridge disks back, we can guarantee the lemma above.
\end{proof}

 Let $\{\beta_1,\beta_2,\beta_3\}$ be a collection of  disjoint three bridge arcs on $\Sigma_{0,6}$ for $T=(B,\tau)$. Now, we consider the subarcs of $\beta:=\beta_1\cup\beta_2\cup\beta_3$ in $P$.
 We assume that there is no bigon so that one of the sides is a subarc of $\beta$ in $P$. If there is a such bigon then we can isotope the subarc to vanish the bigon. We say that $\beta$ is \textit{
dense} in $P$ if  none of two adjacent endpoints of the component of $\beta$ in $\partial P$ belongs to the same $\beta_i$ for $i=1,2,3$. We assume that there are at least two successive adjacent  endpoints of the components of $\beta$ in $\partial P$ contained in the same $\beta_i$ to not satisfy the definition of dense as in Figure~\ref{F6}. We note that two of the successive endpoints connect the puncture directly without passing through the other successive endpoints. Let $c_1$ and $c_2$ be the two subarc of $\beta_i$ explained above. We define the two directions $(+)$ and $(-)$ as in the first diagram of Figure~\ref{F6}.  
If the two punctures appear when the two subarcs go to the opposite direction  we take the new  arc instead of $\beta_i$ as in the second diagram of Figure~\ref{F6}. By Lemma~\ref{L1}, the new arc is also a bridge arc. Moreover, it also clear that the newly selected bridge arc is disjoint with $\beta-\beta_i$.  Similarly, if the two punctures appear when they go to  the same direction as in the third diagram, we take the new bridge arc which is disjoint with the remaining two bridge arcs as in the fourth diagram of Figure~\ref{F6}.
 This procedure is called the \textit{bridge arc replacement}, briefly $\mathbf{BR}$. 
 \begin{figure}[htb]
 \includegraphics[scale=.7]{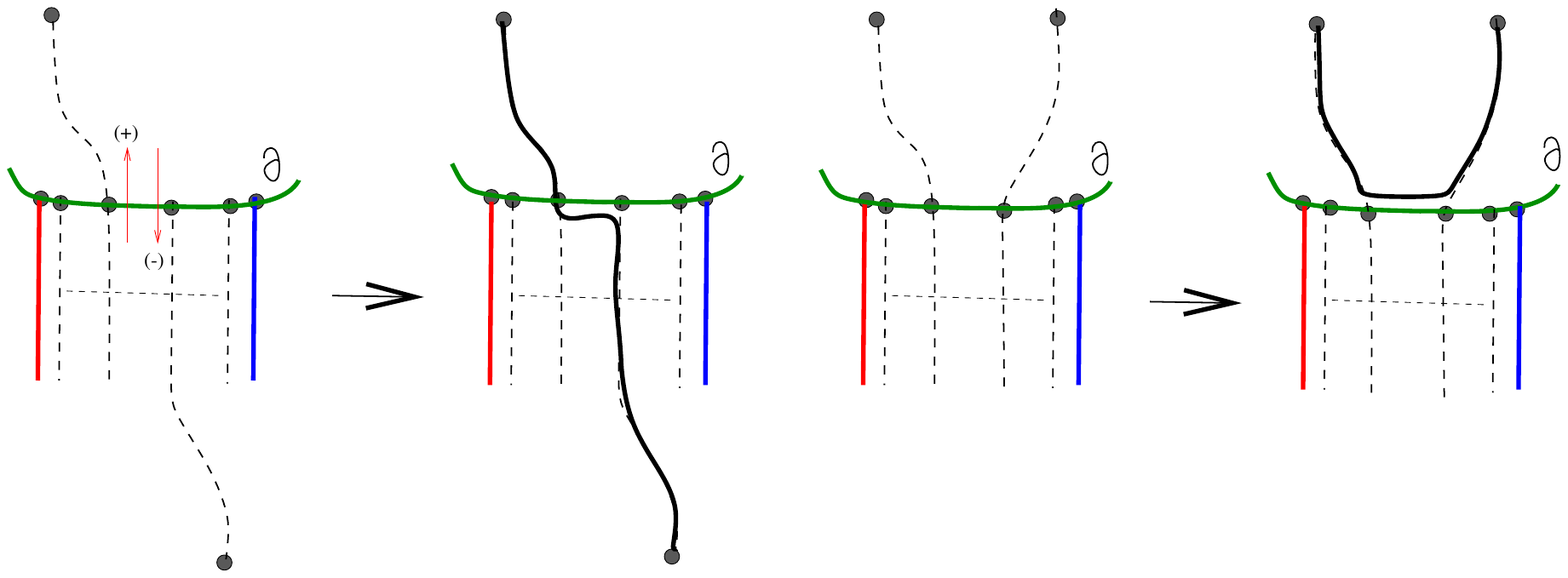} \vskip -450pt
  \caption{Bridge arc replacement, $
  \mathbf{BR}$}\label{F6}
 
 \end{figure}
 
 \begin{Lem}
 Let $\{\beta_1,\beta_2,\beta_3\}$ be a collection of disjoint three bridge arcs on $\Sigma_{0,6}$ of $T=(B,\tau)$. Then there exist a collection of disjoint three bridge arcs $\{\beta_1',\beta_2',\beta_3'\}$ on $\Sigma_{0,6}$ of $T$ which is dense in $P$.
 \end{Lem}
 \begin{proof}
By repeating the procedure $\mathbf{BR}$, we would get the collection which is dense in $P$ since $\mathbf{BR}$ reduces the number of endpoints if the collection is not dense and the number of endpoints are finite.
 \end{proof}
 We claim that the procedure \textbf{BR} can detect the $\infty$-tangle. The following theorem would cover two of the three possible cases.
 \begin{figure}[htb]
\includegraphics[scale=.8]{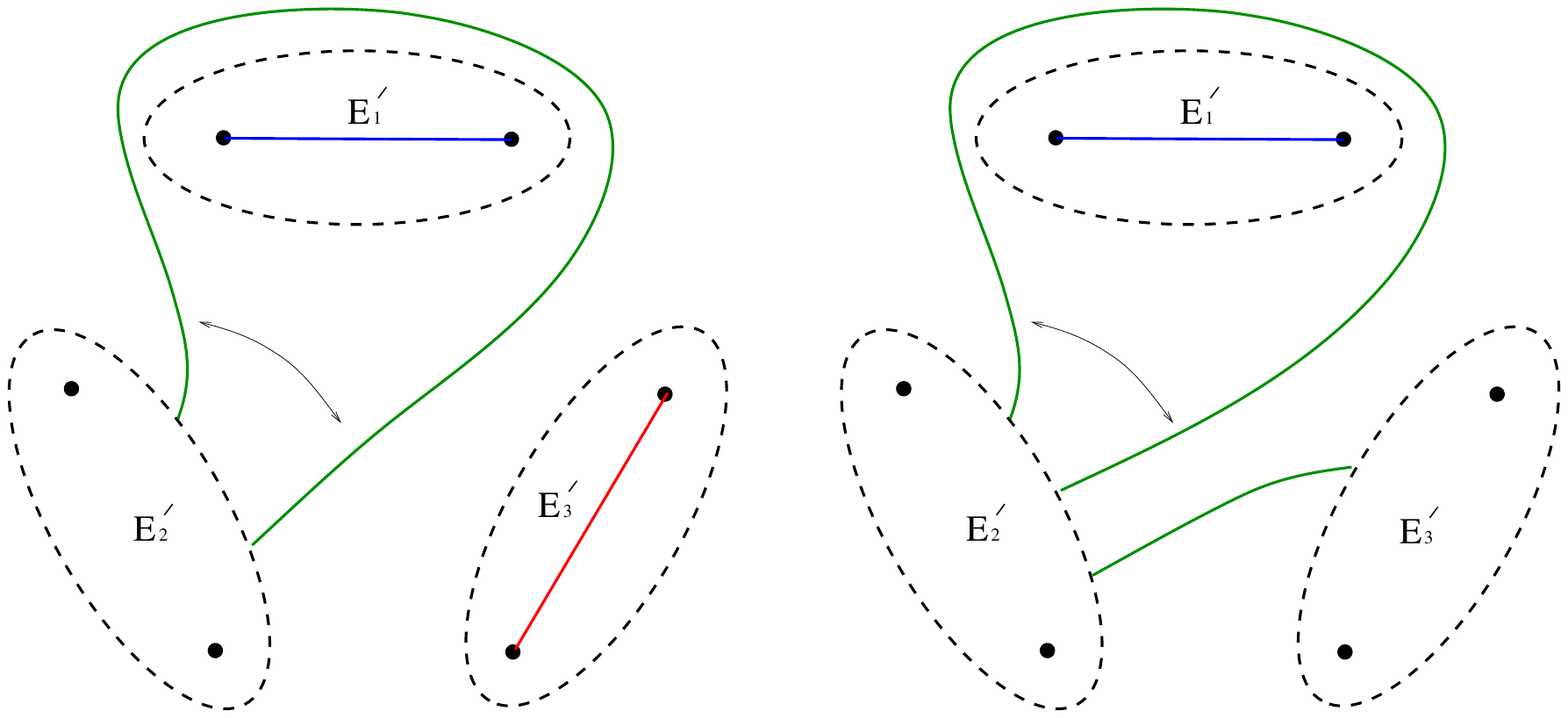}
\vskip -470pt
\caption{A collection of three disjoint simple arcs which is dense in $P$.}\label{F7}
\end{figure}
\begin{Thm}\label{T11}
Let $\mathcal{B}=\{\beta_1,\beta_2,\beta_3\}$ be a collection of disjoint three bridge arcs on $\Sigma_{0,6}$ of $T=(B,\epsilon)$ which is dense in $P$. Then $\mathcal{B}$ is unique up to isotopy. Especially, the three simple closed curves obtained from $\{\beta_1,\beta_2,\beta_3\}$ respectively are $\partial E_1,\partial E_2,\partial E_3$ up to isotopy.
\end{Thm}

\begin{proof}[Proof of the first two cases]
We note that if a simple closed curve obtained from $\beta_i$ is isotopic to $\partial E_j$  then $\beta_i$ is the straight arc connecting the two punctures in $E_j'$ up to isotopy.
We first assume that two of the simple closed curves obtained from $\{\beta_1,\beta_2,\beta_3\}$ are isotopic to $\partial E_i$ and $\partial E_j$ for $i\neq j$. Then we have the first diagram in Figure~\ref{F7}. It is clear that the remaining bridge arc is also isotopic to the straight arc connecting the two punctures in $E_k'$ if the collection is dense in $P$.
Secondly, we assume that one of them is isotopic to $\partial E_i$ for some $i$. Then we have the second diagram of Figure~\ref{F7}. We note that it is impossible to have a wave in terms of $\partial E$ if the collection is dense in $P$ in this case. This implies that there is no wave. So, every simple closed curves obtained from $\{\beta_1,\beta_2,\beta_3\}$ are isotopic to one of $\partial E_j$ for $j=1,2,3$. 
\end{proof}
For the last case, we assume that none of them is isotopic to $\partial E_i$ for some $i$. In order to prove the last case, we use the arguments when we made the algorithm to classify rational $3$-tangles.
\subsection{ Dehn's parametrization of a simple closed curve in $\Sigma_{0,6}$}
\hfill
\break

\begin{figure}[htb]
\includegraphics[scale=.7]{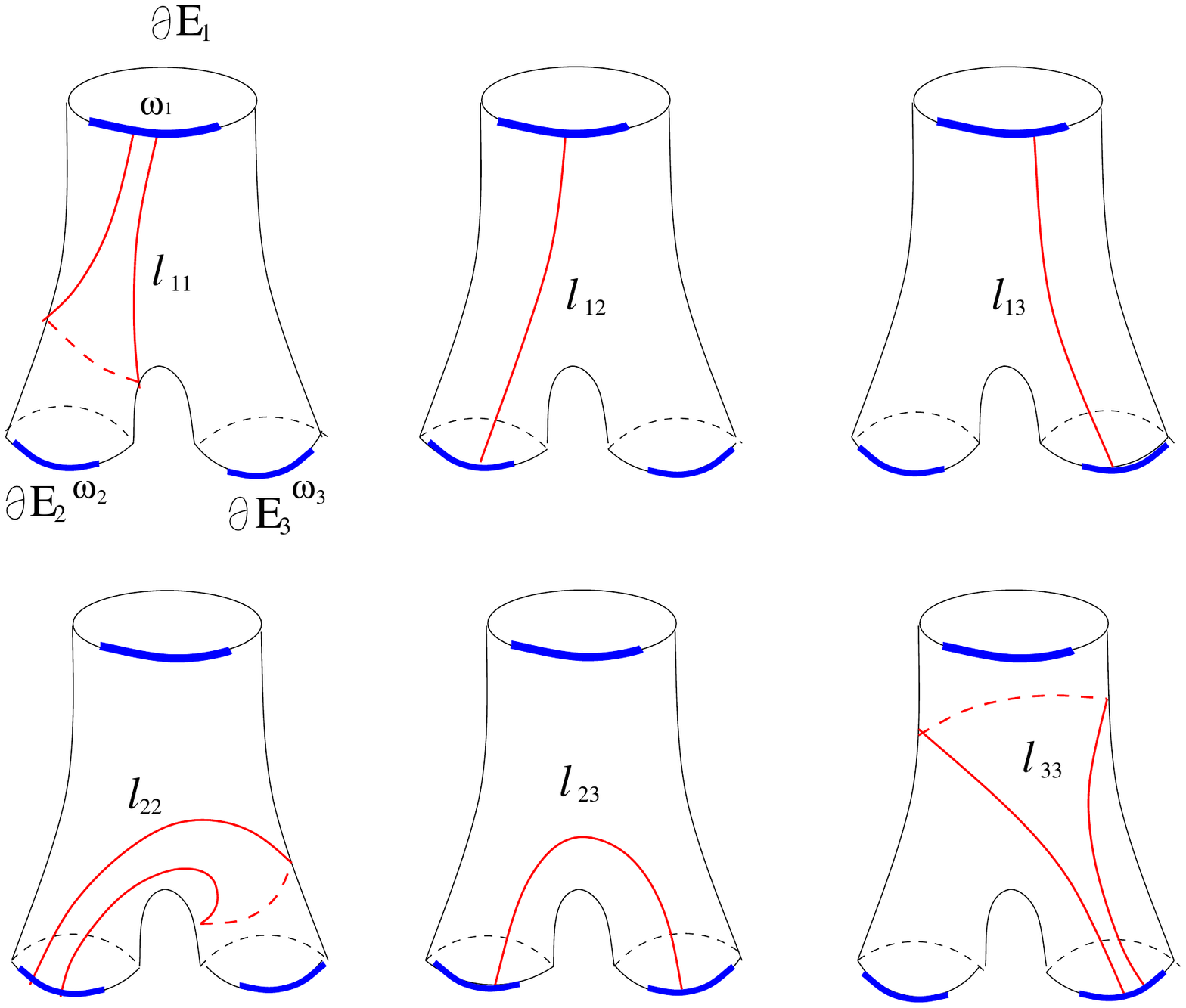}
\vskip -250pt
\caption{Standard arc in $P$}
\label{F2}
 \end{figure}

Let $\gamma$ be a simple closed curve in $\Sigma_{0,6}$. Assume that $\gamma$ meets $\partial E_i$. Then we can isotope $\gamma$ so that it meets with $\partial E_i$  in a special subarc $\omega_i$ of $\partial E_i$ only. We call $\omega_i$ the \textit{window} in $\partial E_i$.Then $\gamma$ has  parallel arcs which are the same type in $P$. Let $l_{ij}$ be the arc type of them as in Figure~\ref{F2} which is called the \textit{standard} arc type. (Refer to \cite{4}.) Let $x_{ij}$ be the number of parallel arcs of the type $l_{ij}$ which is called the \textit{weight} of $l_{ij}$ for $\gamma$. There is a relation between $x_{ij}$ and the geometric intersection number $|\gamma\cap \omega_s|$ as follows.
\begin{Lem}[\cite{1}]\label{T1}
Let $I_s=|\gamma\cap \omega_s|$, where  $s\in\{1,2,3\}$. Then $I_s$   determine $x_{ij}$.
Especially, if $x_{ii}>0$ for some $i$, then $x_{ij} = I_j, x_{ik} = I_k$
and $x_{ii} = {I_i-I_j-I_k\over 2}$. Also, if $x_{ii}=0$  for all $i$, then $x_{ij}={I_i+I_j-I_k\over 2}$ for distinct $i,j,k$. 
\end{Lem}
 We note that $\gamma$ can be determined by $x_{ij}$ and the patterns in $E_i'$. So, we can parametrize  $\gamma$  in $\Sigma_{0,6}$ as follows. \\
\begin{figure}[htb]
\includegraphics[scale=.7]{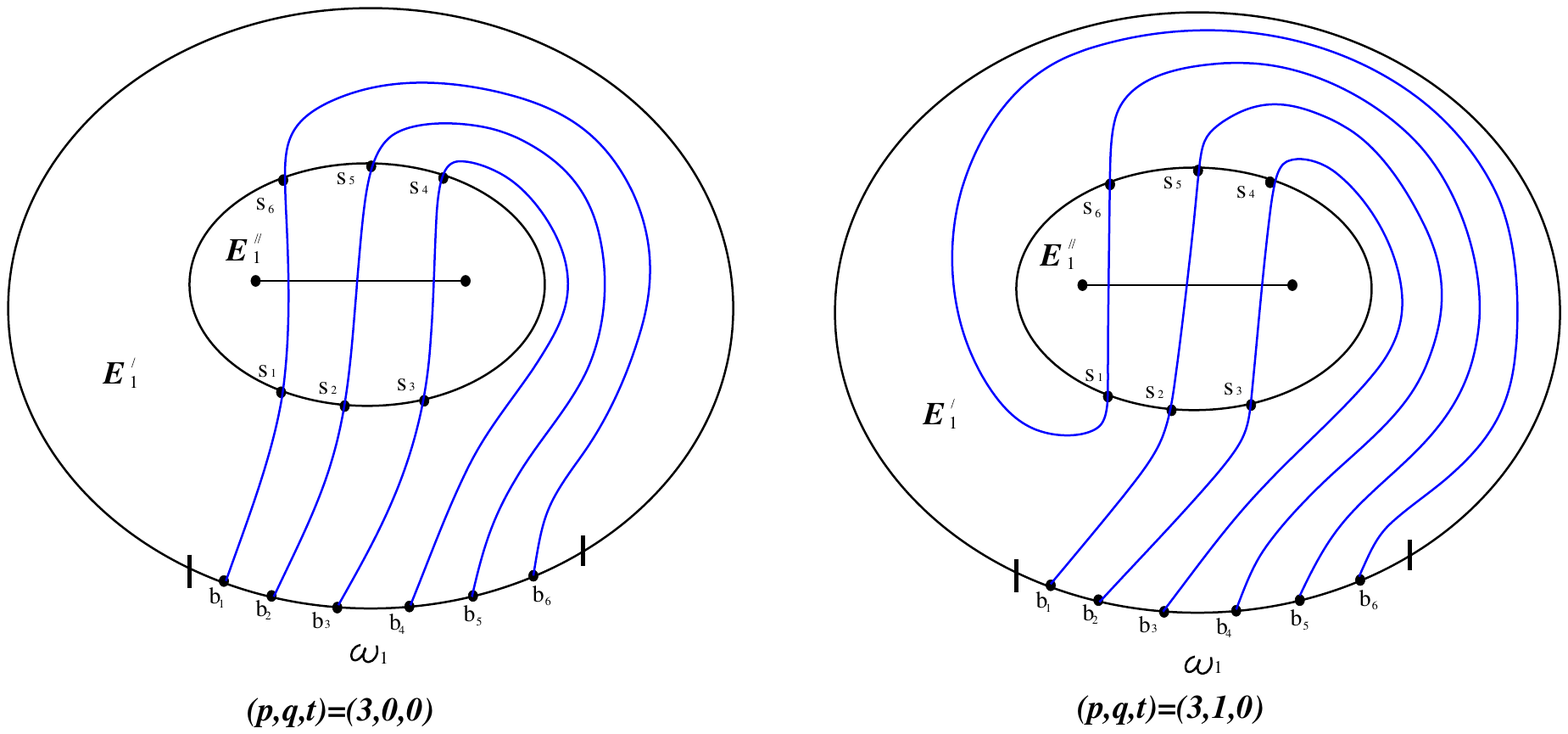}
\vskip -420pt
\caption{parameterization in two punctured disk}
\label{F3}
\end{figure}

 \begin{Thm} [Special case of Dehn's Theorem ~\cite{4}]
There is a parametrization of closed essential 1-manifolds in $\Sigma_{0,6}$ by ordered sequences of nine parameters $(p_1,q_1,t_1,p_2$ $,q_2,t_2,p_3,q_3,t_3).$
 (where $p_i,q_i\in \mathbb{Z}^+\cup \{0\}$ and $t_i\in \mathbb{Z}$ for $i=1,2,3.)$, where  $2p_i=I_{i}$  and $2p_i=J_{i}$. 
\end{Thm}

In the diagrams in Figure~\ref{F3}, $E_i''$ is the inner two punctured disk in $E_i'$ to have only parallel arcs. Then  $q_i$ and $t_i$ determine all the arc types and weights of them in $E_i'$ by determining the connecting pattern between $b_i$ and $s_j$. \\

We are dealing with the simple arc instead of the simple closed curve for proving the main theorem.
We note that the  argument in this section is in terms of simple closed curves not a simple arc. The simple closed curves obtained from the simple arcs can be parameterized by the argument in this section. Moreover,  a simple arc $\beta_i$ is a bridge arc if and only if $\gamma_i$ obtained from $\beta_i$ bounds a compressing disk in $B-\epsilon$.  So, we can discuss the all argument in terms of simple arcs instead of simple closed curves if we consider the exceptional  arc types having a puncture as one of the endpoints of them.

\section{Standard diagram and the proof of the last case}

Recall that $\{\beta_1,\beta_2,\beta_3\}$ is a collection of disjoint bridge arcs on $\Sigma_{0,6}$ for $T=(B,\epsilon)$. 
  In order to analyse $\beta_i$, we set $\Sigma_{0,6}$ with the three disjoint non-parallel simple closed curves $\partial E_1, \partial E_2$ and $\partial E_3$.  Let $P'=\partial B-(\cup_{i=1}^2 E_i'').$
Now, take equators $e_i$ for each $E_i''$ as in Figure~\ref{F4}. Then $e_i$ divides $E_i''$ into $E_i''^{+}$ and $E_i''^{-}$ as in Figure~\ref{F4}. We note that $\partial E_i$ is isotopic to $\partial E_i''$ for $i=1,2,3$. Now, we introduce a lemma which helps us to realize $\beta$ in a certain formed diagram called \textit{standard diagram}. 
 \begin{Lem}[Lemma $6.1$ in \cite{1}]\label{L3}
 Let $N$ be a compressing disk in $B-\epsilon$ and $h$ be the homeomorphism obtained from the clockwise half Dehn twist supported on $N'$ which is the 2-punctured disk in $\Sigma_{0,6}$ so that $\partial N=\partial N'$ . Then a simple closed curve $\delta$ bounds a compressing disk in $B-\epsilon$ if and only if  $h(\delta)$ bounds a compressing disk in $B^3-\epsilon$. 
 \end{Lem}
   
We note that Lemma~\ref{L3} implies that a simple arc $\beta_i$ is a bridge arc in $B-\epsilon$ if and only if $h(\beta_i)$ is a bridge arc in $B-\epsilon$.

\begin{figure}[htb]
\includegraphics[scale=.6]{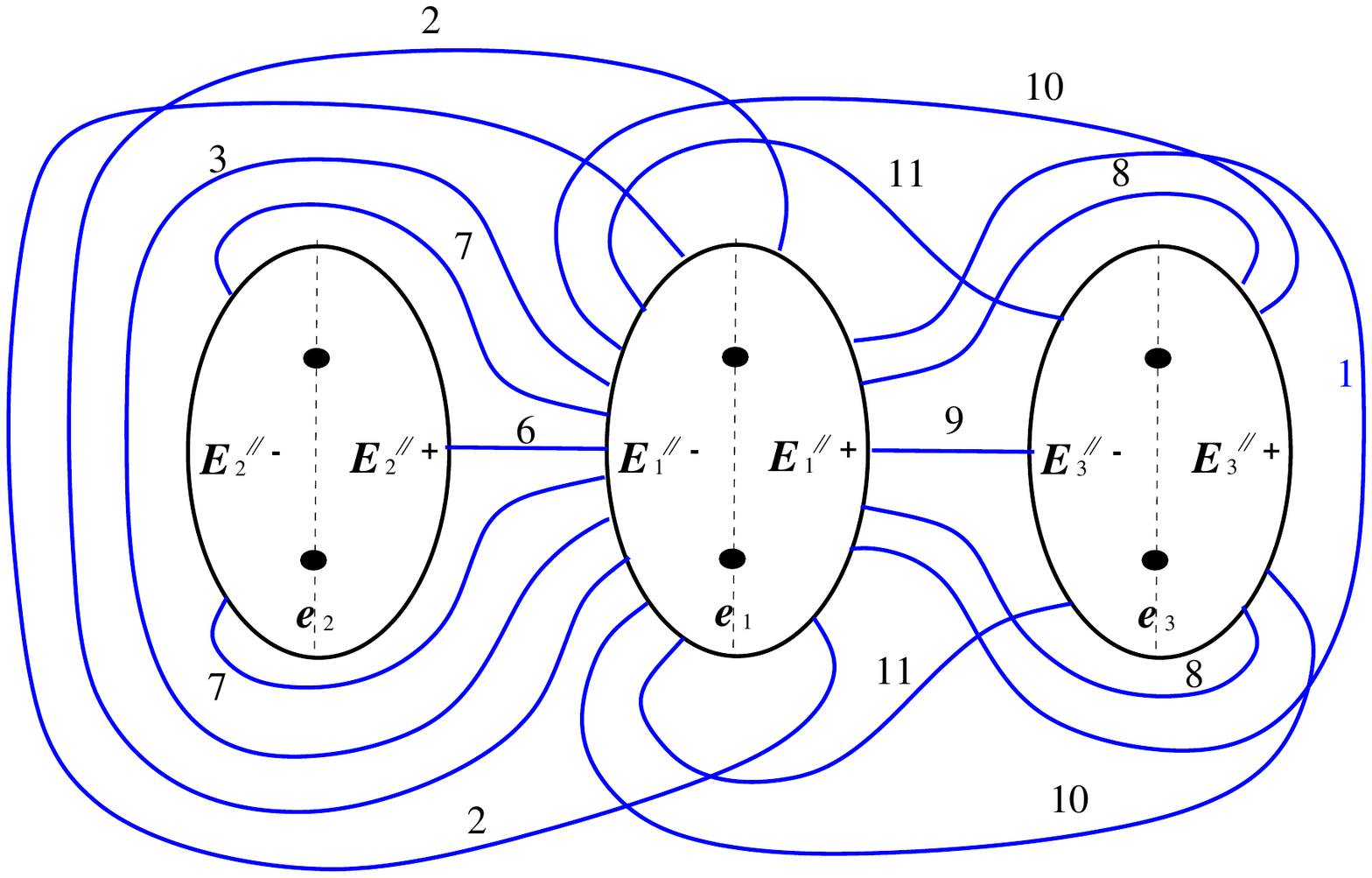}
\vskip -290pt
\caption{Standard diagram}\label{F4}
 \end{figure}
 
Let $\gamma_i$ be the simple closed curve obtained from $\beta_i$ which can be parametrized by an ordered sequence of integers $(p_1,q_1,t_1,p_2,q_2,t_2$ $,p_3,q_3,t_3)$.
We now investigate the  simple closed curve $\gamma_i'$ (it may be in a different isotopy class in $\Sigma_{0,6}$) so that every component of $\gamma_i'\cap P'$ is isotopic to one of the given real arc types  in $P'$ of one of the two diagrams as in Figure~\ref{F4} by taking a properly chosen $t_i'$ instead of $t_i$ by the following lemma.

\begin{Lem}[Lemma $8.2$ in \cite{1}]\label{L7}
Suppose that $\gamma$ is parametrized by $(p_1,q_1,t_1,p_2,q_2,t_2,p_3,q_3,t_3)$ and it has $l_{11}$ arc type in $P$.
Then there exist simple closed curves $\gamma'$ which is parameterized by $(p_1,q_1,t_1',p_2,q_2,t_2',p_3,q_3,t_3')$ for some $t_i'\in \mathbb{Z}$
so that all components of $\gamma'\cap P'$ are isotopic to one of the arcs in the  diagram.
\end{Lem}

We say that a simple closed curve $\gamma$ is in \textit{standard position} if all componet of $\gamma\cap P'$ are isotopic to one of the arcs in the diagram of Figure~\ref{F4}. The diagram in Figure~\ref{F4} is called \textit{standard diagram}.
Let $\beta_i'$ be the simple arc so that $\gamma_i'$ is the simple closed curve obtained from $\beta_i'$.  We note that the components of $\beta_i'\cap P'$ are also carried by the arc types in the standard diagram. In order to see this, it is enough to investigate the component in $P'$ which connects the puncture in $E_i''$ directly as in Figure~\ref{F5}. We note that $\gamma$ has the two components $(a)$ and $(b)$ and $\beta$ has the component ($b'$) which is parallel to the component $(b)$. 
The original lemma  is for a simple closed curve not a multiple of simple closed curves. However, there is no difference between a simple closed curve and multiple simple closed curves when we apply the above lemma. It works for $\beta'=\beta_1'\cup\beta_2'\cup \beta_3'$ as well. We  note that the changes of the twisting number $t_i$ would preserve the density since the half Dehn twist along $E_i''$ just rotate  the position of the all endpoints on $\partial E_i''$ by $180^\circ$.
We also note that the trivial rational $3$-tangle $\epsilon$ in $B$ is preserved when we apply the half Dehn twists supported on $N'$ defined in Lemma~\ref{L3}. For  additional modifications of $\beta$, we consider the following two lemmas and the theorem.

\begin{figure}[htb]
\includegraphics[scale=.4]{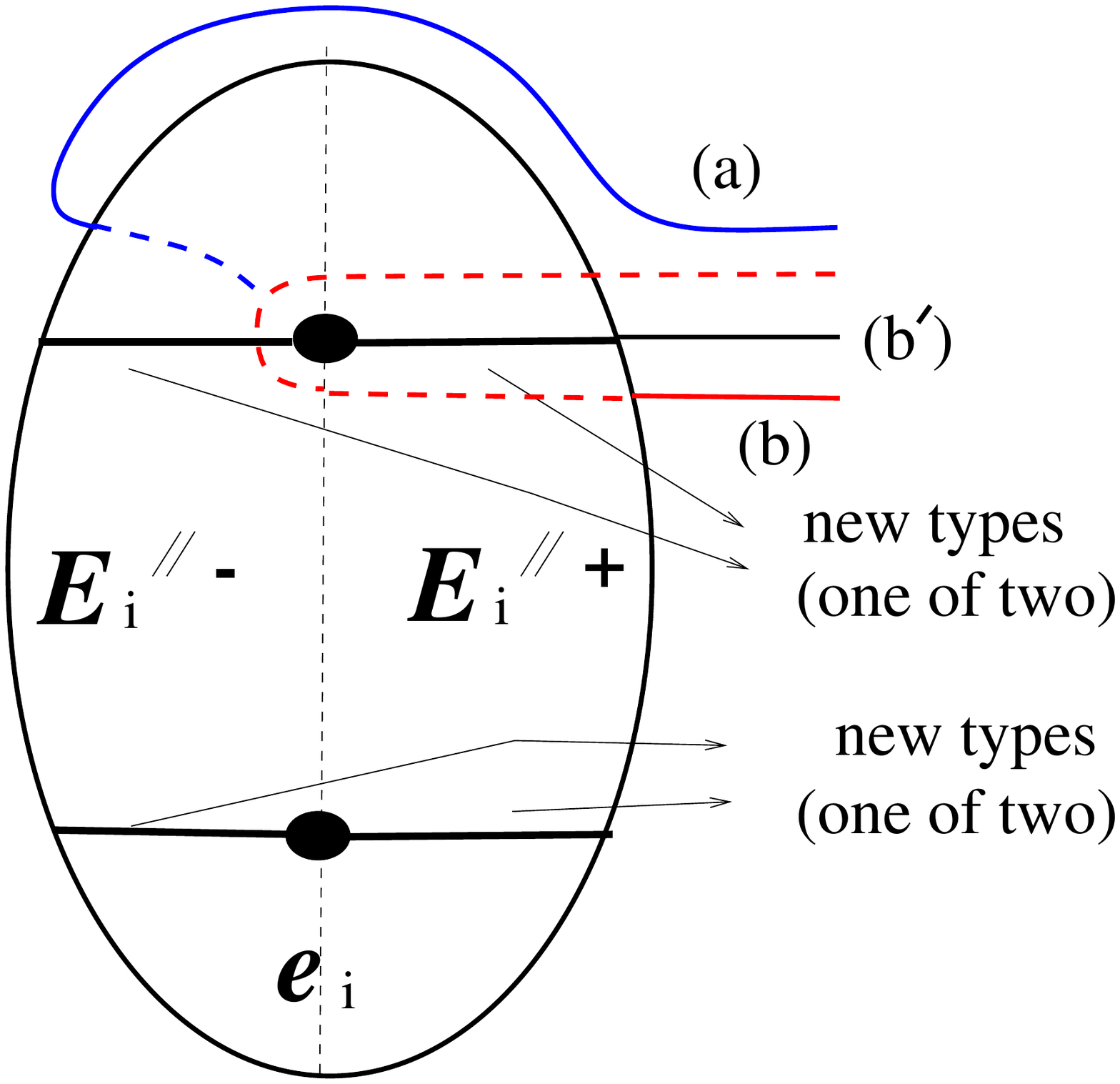}\vskip -130pt
\caption{}\label{F5}
\end{figure}

\begin{Lem}[Lemma $9.1$ in \cite{1}]\label{L8}
 If $A$ is a compressing disk in $B^3-\epsilon$, then  $(\delta_1\delta_2^{-1})^{\pm 1}(\partial A)$ also bounds a compressing disk in $B^3-\epsilon$. (where $\delta_1$ and $\delta_2$ are the clockwise half Dehn twists supported on two punctured disks $C_1$ and $C_2$ respectively as in Figure~\ref{6}.)
\end{Lem}
\begin{figure}[htb]
\includegraphics[scale=.6]{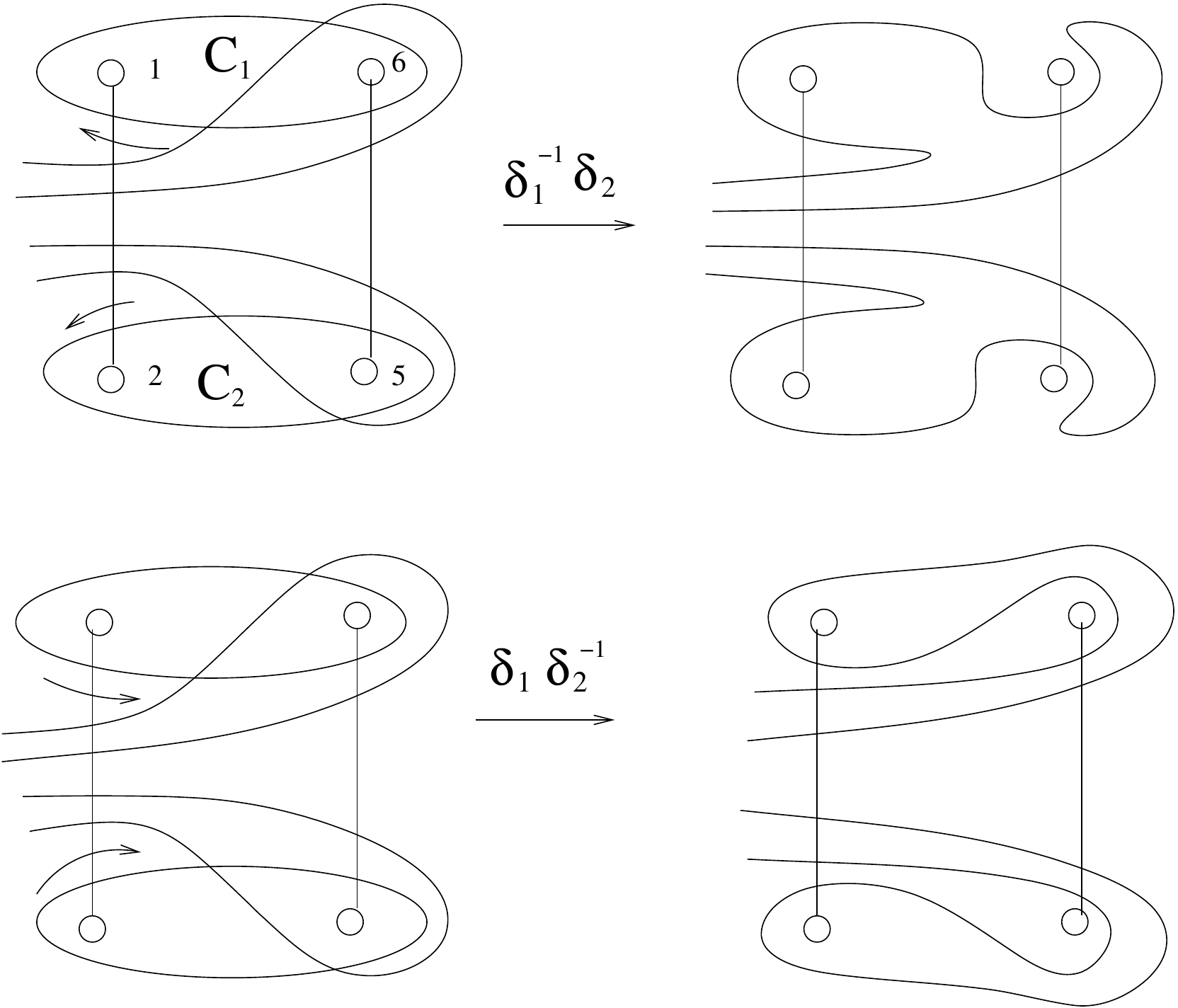}
\caption{Half Dehn twists to preserve the boundness of a compressing disk in $B^3-\epsilon$}\label{6}
 \end{figure}

\begin{Lem}[Lemma $9.2$ in \cite{1}]\label{L9}
Suppose that $ A$ is a compressing disk in $B^3-\epsilon$. Let $\delta_3$ be the counter clockwise half Dehn twist supported on the 2-punctured disk $E_4'$ as in Figure~\ref{7}.
Then $\delta_3(\partial A)$ also bounds a compressing disk $\delta_3(A)$ in $B^3-\epsilon$.
\end{Lem}

\begin{figure}[htb]
\includegraphics[scale=.7]{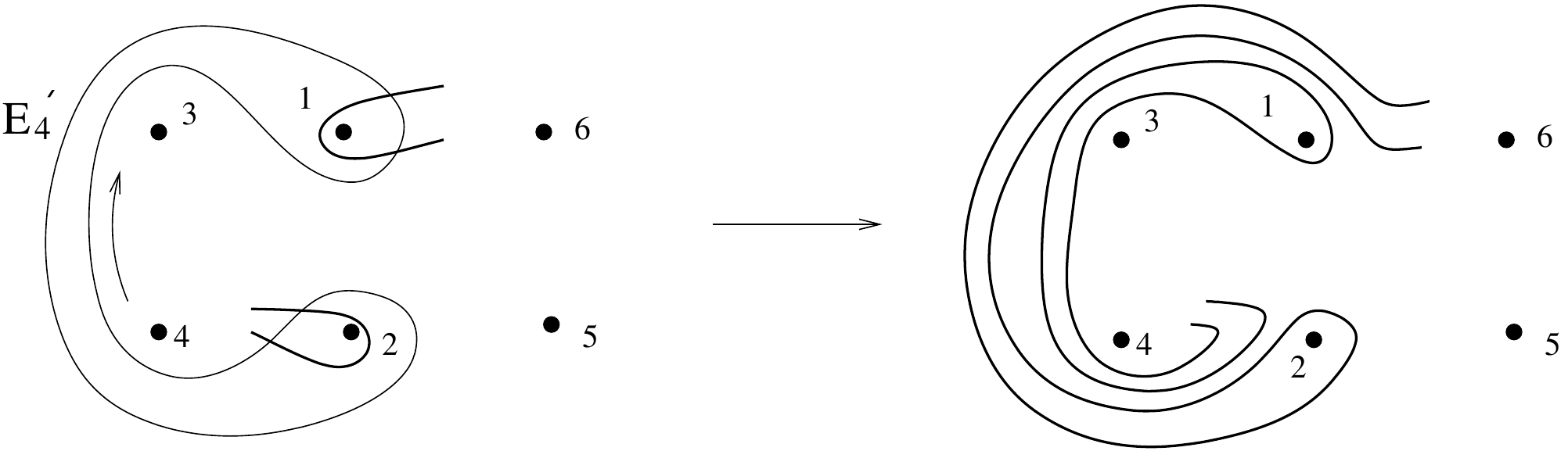}
\caption{A half Dehn twists to preserve the boundedness of a compressing disk in $B^3-\epsilon$}\label{7}
 \end{figure}
 
\begin{Thm}[Theorem $9.3$ in \cite{1}]\label{T10}
Suppose that $A$ is a compressing disk in $B^3-\epsilon$ and $\partial A$  is in standard position in $P'$.
Then one of the half Dehn twists  $(\delta_1\delta_2^{-1})^{\pm 1}$ and $\delta_3^{\pm 1}$ reduces the minimal intersection number of $\partial A$ with $\partial E$.
\end{Thm}
 
The following lemma is the main idea to prove the last case of the main theorem.
 \begin{Lem}\label{L2}
  Let $\{\beta_1,\beta_2,\beta_3\}$ be the collection of  disjoint simple arcs which are bridge arcs in $B-\epsilon$. We assume that $\beta=\beta_1\cup\beta_2\cup\beta_3$ is in standard position and all of them have a wave in terms of $\partial E''$. Then, there exist two adjacent endpoints of $\beta\cap P'$ in $\partial E_i''$  for some $i$ which belong to the same $\beta_j$ for $j=1,2,3$.
 \end{Lem}

 \begin{proof}
 
   \begin{figure}[htb]
 \includegraphics[scale=.9]{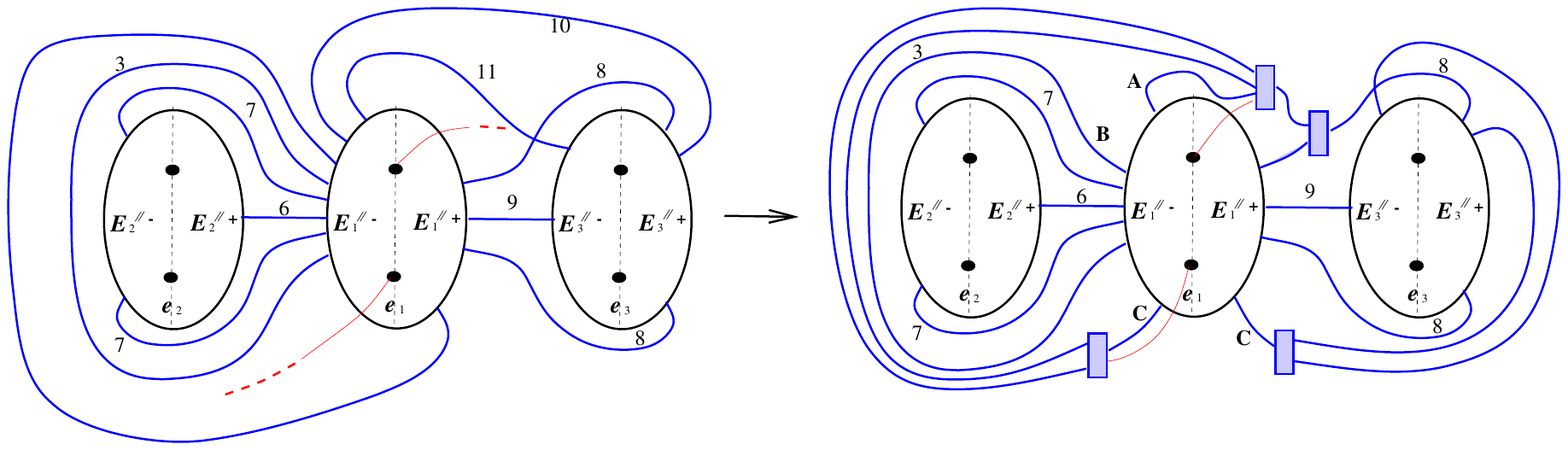}
 \vskip -620pt
 \caption{}\label{F9}
 \end{figure} 
   \begin{figure}[htb]
 \includegraphics[scale=.9]{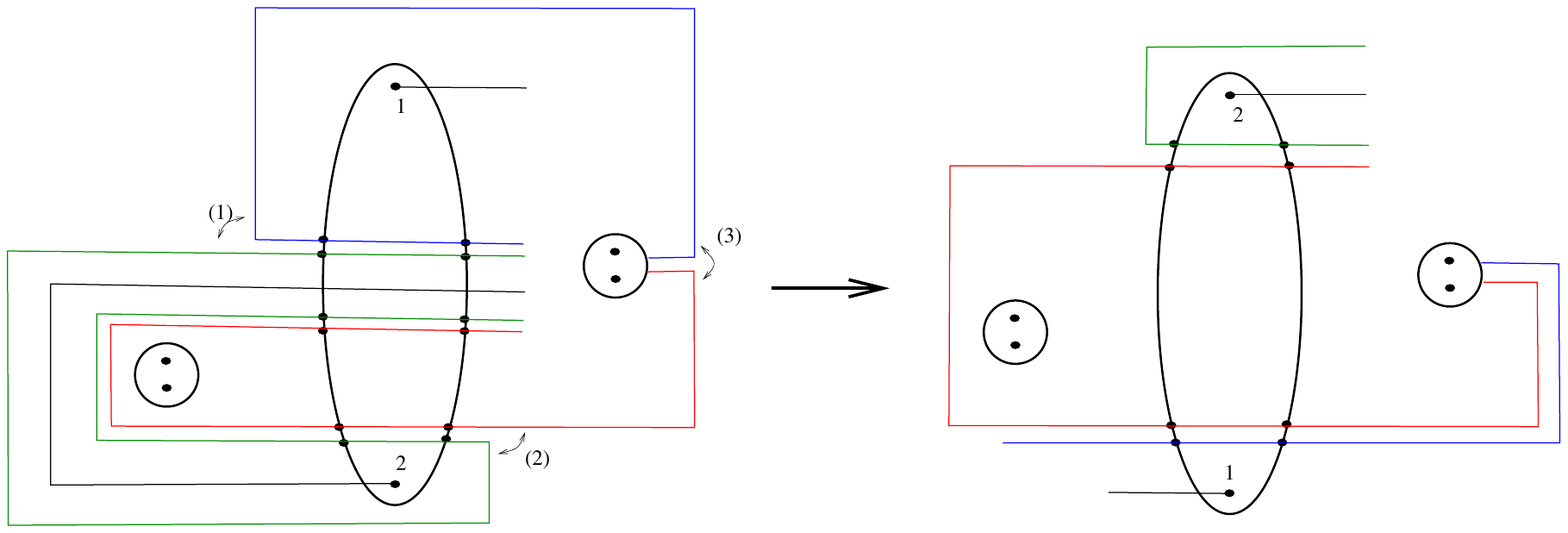}
 \vskip -580pt
 \caption{}\label{F15}
 \end{figure}

 By the assumption that all of $\beta_i$ have a wave,  without loss of generality, we assume that $\beta$ has  subarcs of type 1 or type 3 as in the first diagram of Figure~\ref{F9} and Figure~\ref{F10}. By applying the  homeomorphism $\delta_3$ in Lemma~\ref{L9} to $\beta$, we can remove the subarcs of type 2 as in the second diagram of Figure~\ref{F9} and Figure~\ref{F10} respectively.  Let $m_k$ be the weights of each type $k$. In the second diagram of Figure~\ref{F9}, we note that $\textbf{A}=m_2, \textbf{B}=m_3-m_2$ and $\textbf{C}=m_{10}+m_{11}$.
 In the right diagram of Figure~\ref{F10}, we have  $\textbf{A}=m_2$ and \textbf{B}=$m_3-m_2$. Also,  in the bottom diagram of Figure~\ref{F10}, we have $\textbf{A}=m_2-m_3, \textbf{B}=m_3$ and $\textbf{C}=m_{2},\textbf{D}=m_1-m_2, \textbf{E}=m_{7_2}+m_3-m_2$, where ${7_2}$ is lower track of the type $7$'s in the standard diagram.  For the detail of this weight changes  in Figure~\ref{F9} and Figure~\ref{F10}, refer to Theorem 9.3 in \cite{1}. We note that the homeomorphism in each  step reduces the total weight of the subarcs in $P'$ by Theorem~\ref{T10}.\\
 
  We claim that the collection of the new bridge arcs is dense in $P'$. In order to show the claim for the case described in Figure~\ref{F9}, consider the diagrams in Figure~\ref{F15}. It is enough to investigate the pattern of endpoints in $\partial E_1''$ since the endpoints in $\partial E_2''$ and $\partial E_3''$ are fixed. Since $\beta$ is dense in $P'$, the depicted two adjacent arcs in the three positions $(1),(2)$ and $(3)$ should be a part of distinct bridge arcs. This implies that the right diagram does not violate the rule to satisfy the density as well. It is not difficult to see that the other adjacent endpoints in $\partial E_1''$ except the indicated adjacent endpoints by the parallel arcs also keep the rule for the density since the patterns are the same with the previous corresponding patterns.
   \begin{figure}[htb]
 \includegraphics[scale=.9]{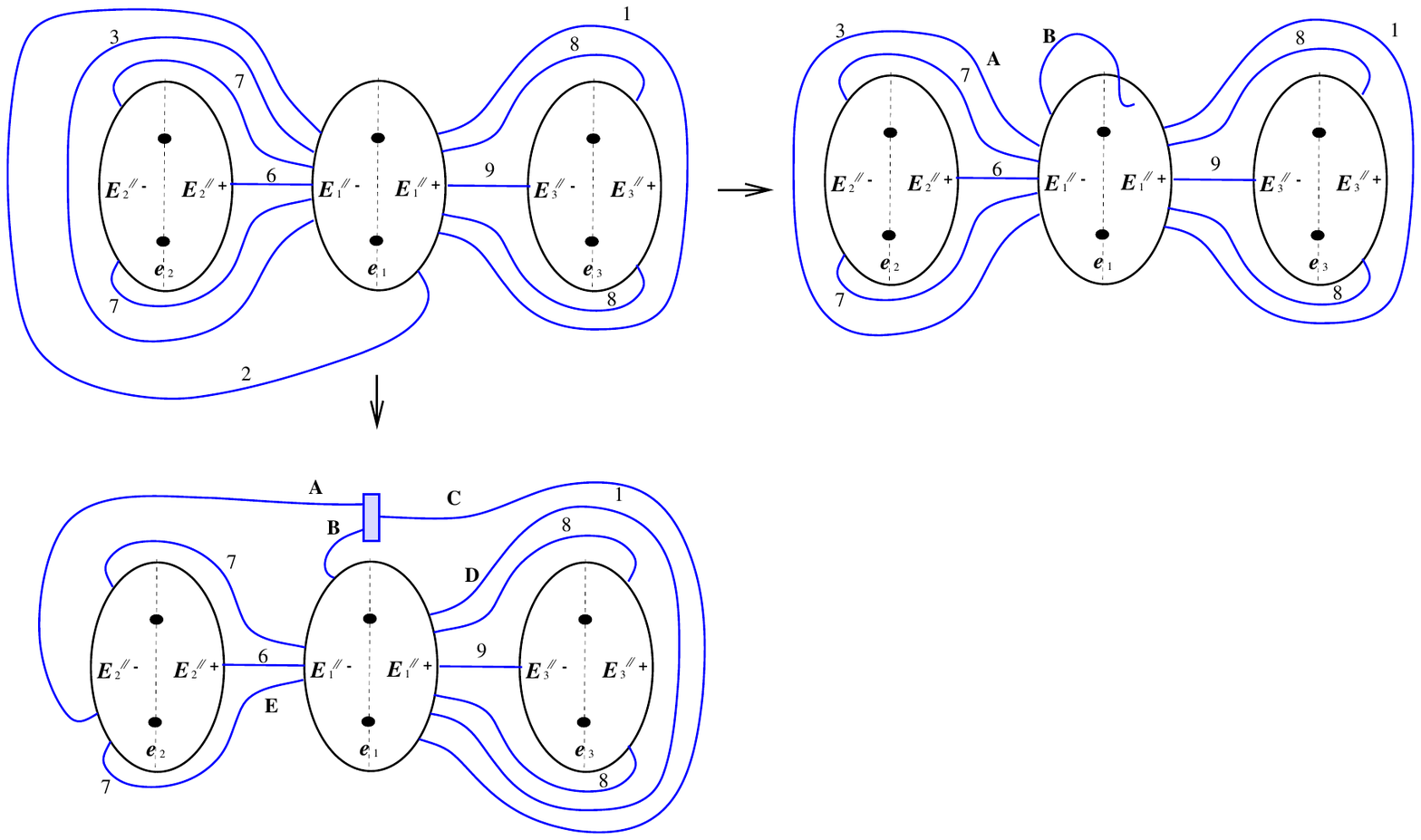}
 \vskip -450pt
 \caption{}\label{F10}
 \end{figure}
  \begin{figure}[htb]
 \includegraphics[scale=.9]{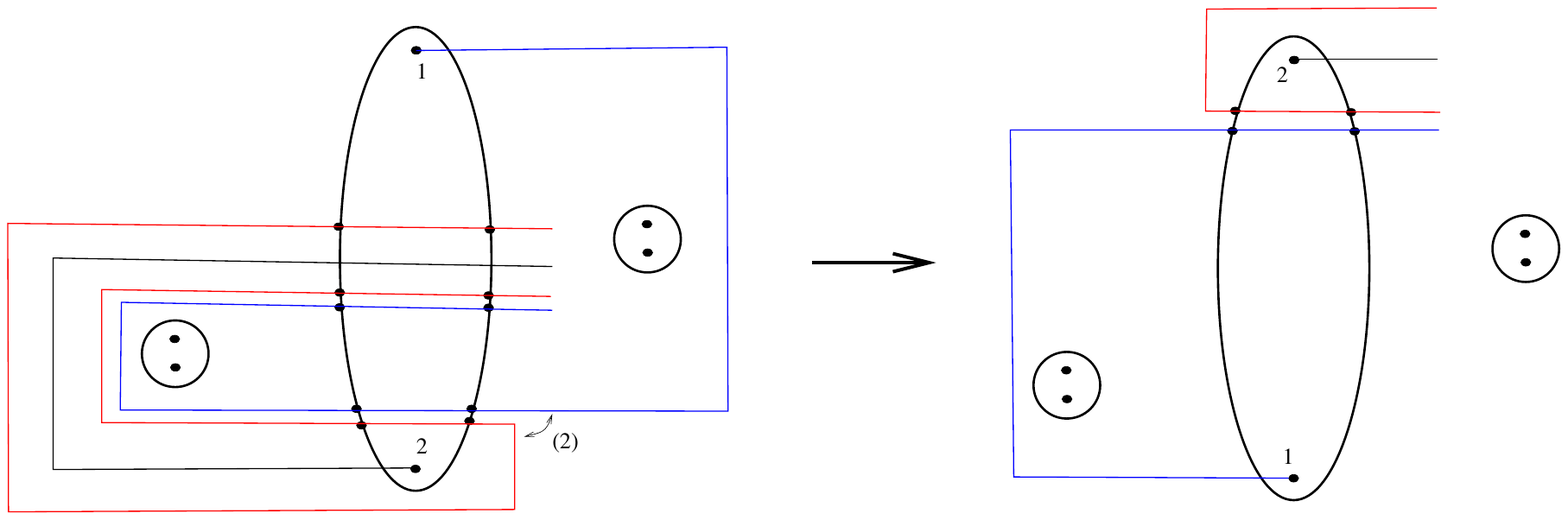}
 \vskip -550pt
 \caption{}\label{F16}
 \end{figure}
  
  Similarly, we note that the collections of the modified bridge arcs described in Figure~\ref{F10} are dense by Figure~\ref{F16}.  Finally, we would get the four diagrams as in Figure~\ref{F11} which do not have the type $2$.  First, we assume that both of $m_1$ and $m_3$ are positive as in the diagram $(A)$. Then, there is two adjacent endpoints in $\partial E_1''$ since the two punctures in $E_1''$ should be the two endpoints of the same $\beta_i$ for some $i=1,2,3$.

 \begin{figure}[htb]
 \includegraphics[scale=.9]{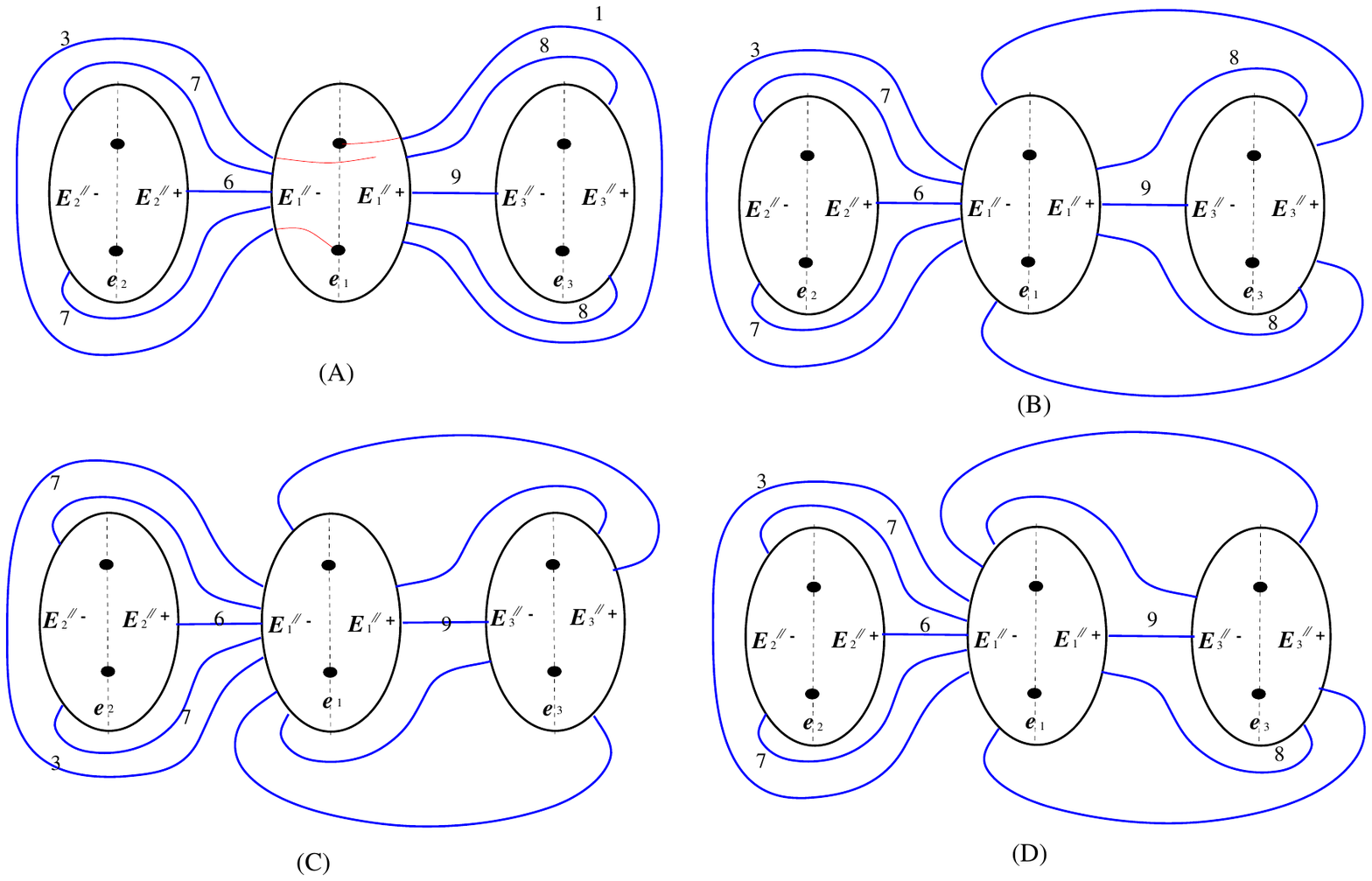}
 \vskip -450pt
 \caption{}\label{F11}
 
 \end{figure}
For the rest of the cases, we apply a multiple of the homeomorphism $(\delta_1^{-1}
 \delta_2)^{\pm 1}$ to have a simpler diagram as follows. For the diagram $(B)$, we apply a multiple of the homeomorphism $(\delta_1^{-1}
 \delta_2)$ to have the rightmost diagram in Figure~\ref{F15}.  We note that the homeomorphism $(\delta_1^{-1}
 \delta_2)$ preserves the density.
 \begin{figure}[htb]
 \includegraphics[scale=.9]{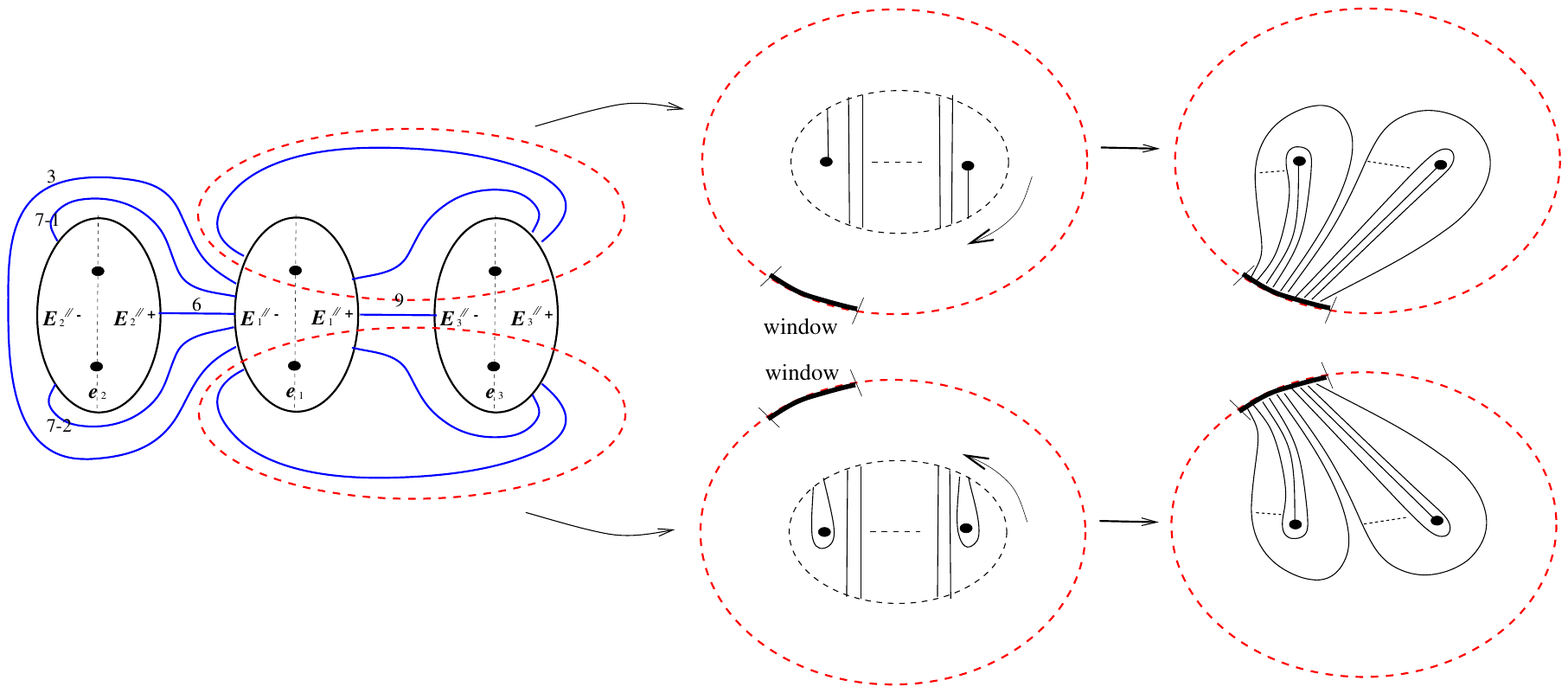}
 \vskip -540pt
 \caption{}\label{F17}
 
 \end{figure}

 \begin{figure}[htb]
 \includegraphics[scale=.9]{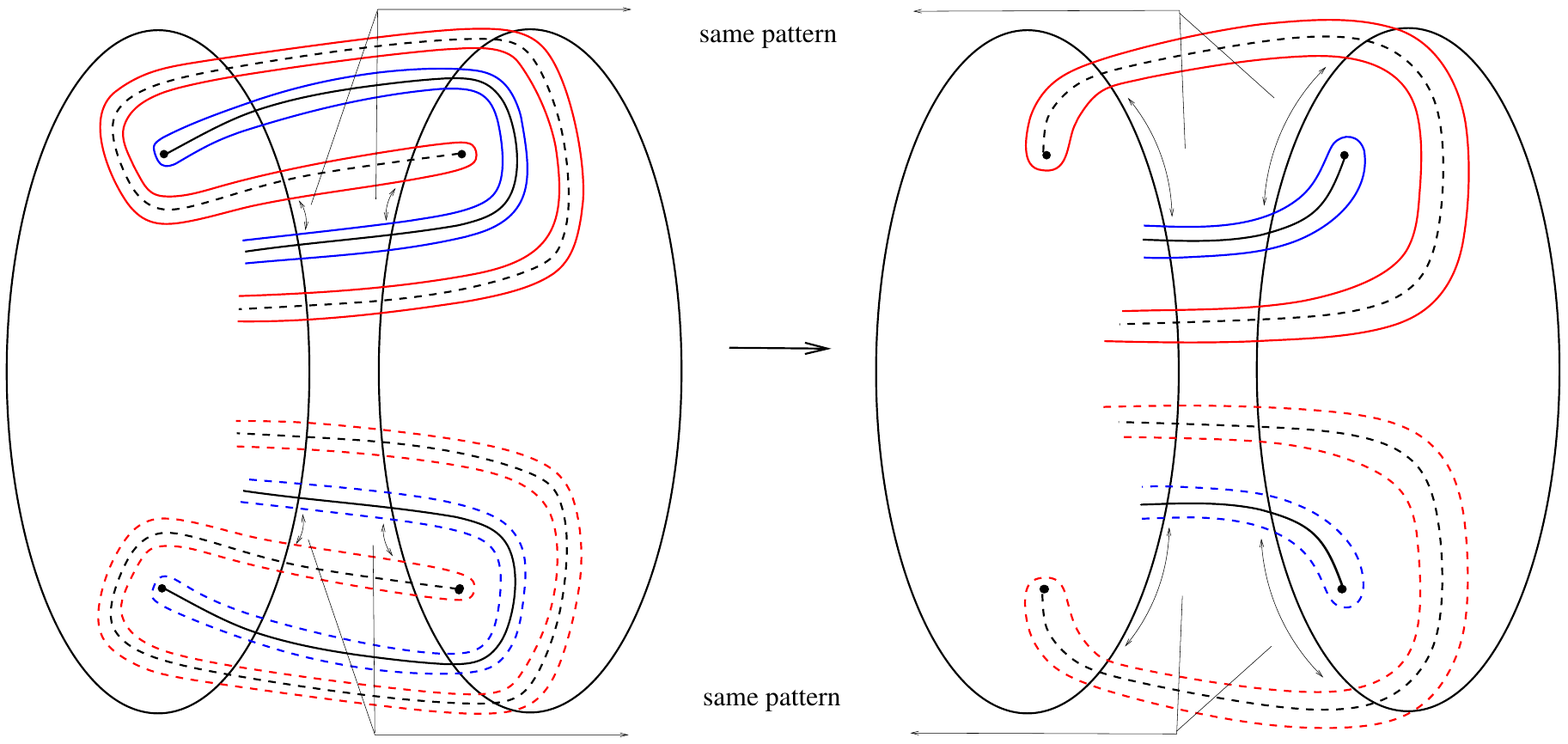}
 \vskip -520pt
 \caption{}\label{F13}
 
 \end{figure} 
  
 \begin{figure}[htb]
 \includegraphics[scale=.9]{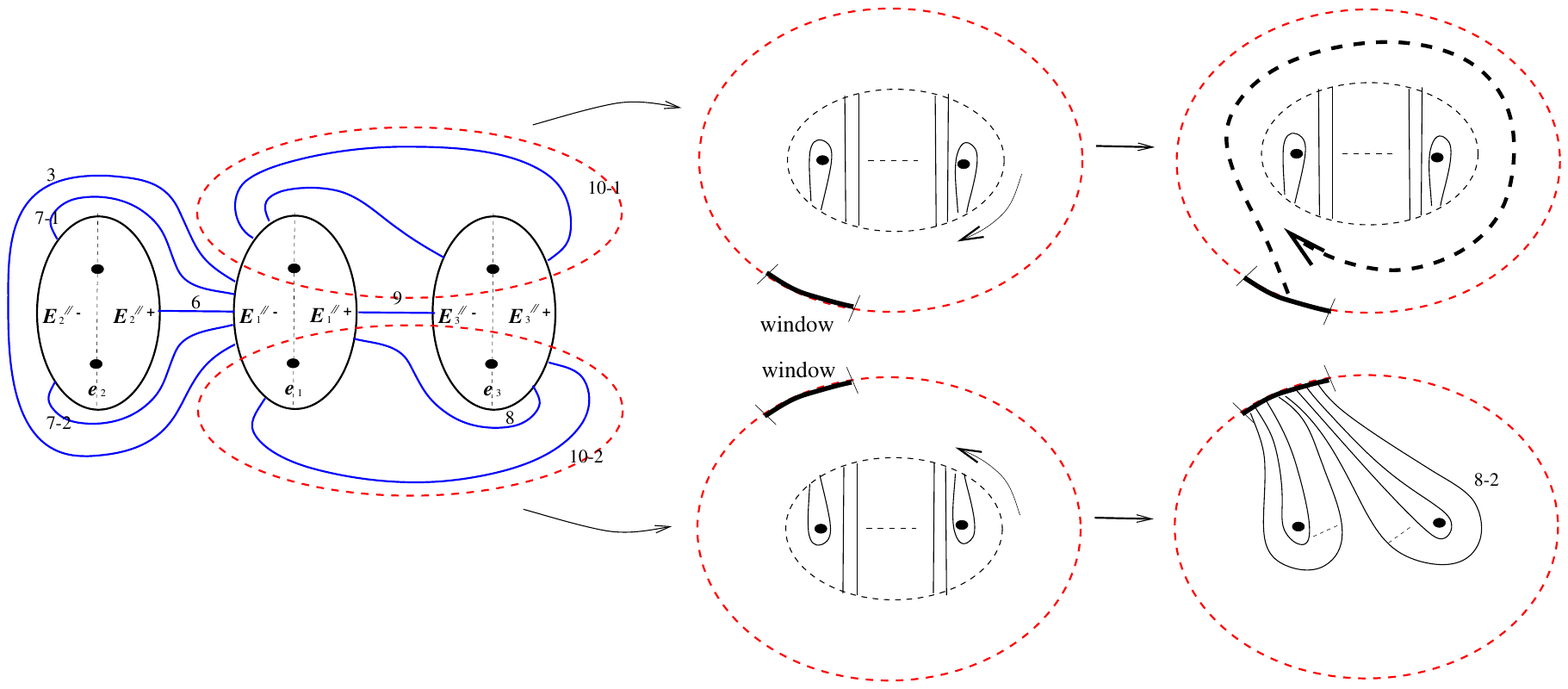}
 \vskip -550pt
 \caption{}\label{F12}
 
 \end{figure}
  
 \begin{figure}[htb]
 \includegraphics[scale=.9]{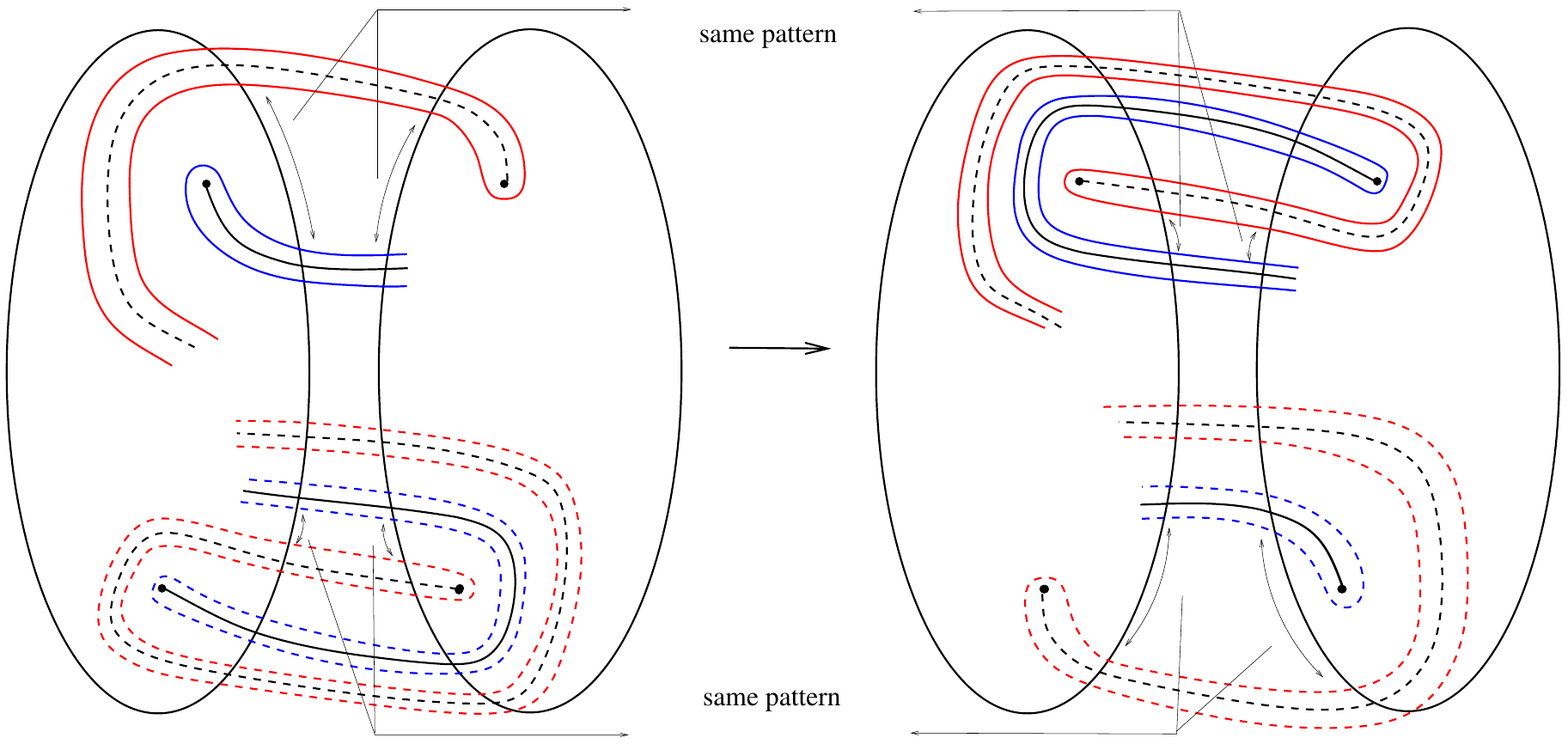}
 \vskip -520pt
 \caption{}\label{F14}
 
 \end{figure}

 Now, consider the cases $(C)$ and $(D)$. Since the cases are symmetric, it is enough to figure out the case $(D)$.  We apply the homeomorphism $(\delta_1^{-1}\delta_2)$ to the diagram $(D)$ until we get the rightmost diagrams   in Figure~\ref{F12}.  We  note that the homeomorphism $(\delta_1^{-1}
 \delta_2)$ preserves the density as in Figure~\ref{F14}. If we have the diagram having the rightmost two diagrams, the diagram has the type $2$. We repeat the procedure to remove the type $2$. We note that the procedure preserves the density. Since the total weights of the subarcs in $P'$ is finite and the procedures mentioned in this proof reduce the total weights, we eventually get a collection of bridge arcs so that at least one of them is isotopic to the straight arc connecting the two punctures in $E_i''$ for some $i\in\{1,2,3\}$. So, it is enough to consider  the diagrams right before the final diagram. Then we have the two diagrams $(a)$ and $(b)$ as in Figure~\ref{F18} if we ignore symmetric cases. We note that if we apply $(\delta_1^{-1}
 \delta_2)$ to the diagram $(a)$ then we have the bridge arc which connects the two punctures in $E_3''$. Also, we apply $\delta_3^{-1}$ to the diagram $(b)$ then we have the bridge arc which connects the two punctures in $E_1''$.
In $(a)$, assume that $\beta_1$ is isotopic to the red arc connecting the two punctures in $E_1''$. Since $\beta_2$ and $\beta_3$ also have a wave, it is clear that there are two adjacent endpoints in $\partial E_2''$ which belong to the same $\beta_j$ for some $j=2,3$. Similarly, in $(b)$, there are two adjacent endpoints in $\partial E_2''$ which belong to the same $\beta_j$ for some $j=2,3$. This implies that either $\delta_3^{\pm1}$ or 
  $(\delta_1^{-1}
 \delta_2)^{\pm 1}$ cannot preserve the density. This violates the argument above.
\begin{figure}[htb]
\includegraphics[scale=.9]{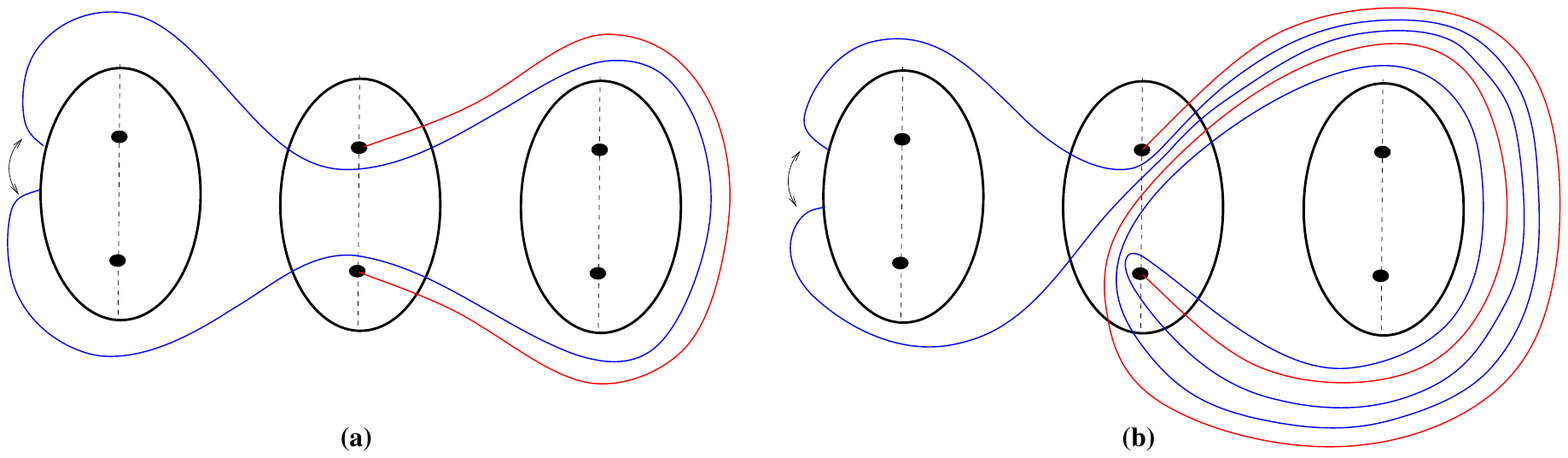}
\vskip -600pt
\caption{}\label{F18}
\end{figure}
 
 \end{proof}
 Now, we can complete the proof of Theorem~\ref{T11}.
\begin{proof}[Proof of the last case]
By Lemma~\ref{L2}, it is not possible to have a collection of disjoint three bridge arcs so that at least one of them has a wave. Therefore, if $\beta$ is dense in $P'$ then none of $\beta_i$ has a wave. This implies that the collection of the bridge arcs should be isotopic to the collection of $\partial E_1,\partial E_2$ and $\partial E_3$.
\end{proof}

 \section*{Acknowledgements}

 This research was supported by Basic Science Research Program through the National Research Foundation of Korea (NRF)  funded by the the Ministry of Education (NRF-2020R1I1A1A01052279 and NRF-2020R1A2C1A01009184)

\end{document}